\documentclass[10pt]{amsproc}
\usepackage{amssymb,amsmath, amsthm, amsfonts}
\usepackage{mathrsfs}
\usepackage[all]{xypic}
\usepackage{amsmath,amssymb,mathrsfs
}
\usepackage{mathrsfs,color,bbold,tikz}
\usepackage[all]{xypic}
\usepackage{alphalph}

\newcommand{\luk}{\L u\-ka\-si\-e\-w\-icz}

\newtheorem{thm}{Theorem}[section]
\newtheorem{lem}[thm]{Lemma}
\newtheorem{cor}[thm]{Corollary}
\newtheorem{prop}[thm]{Proposition}

\theoremstyle{definition}
\newtheorem{definition}[thm]{Definition}
\newtheorem{remark}[thm]{Remark}
\newtheorem{example}[thm]{Example}

\newcommand{\Mod}{{\rm Mod}}

\newcommand{\relint}{\mathsf{ri}}
\newcommand{\relb}{\mathsf{rb}}
\newcommand{\conv}{\mathsf{co}}
\newcommand{\cl}{\mathsf{cl}}
\newcommand{\ext}{\mathsf{ext}}

\newcommand{\free}{{\bf F}}

\begin{document}

\title[]{Three characterizations of strict coherence on infinite-valued events}

 \author[T. Flaminio]
        {Tommaso Flaminio}
\address{IIIA - CSIC\\ Campus de la Universidad Aut\`onoma de Barcelona s/n, 08193 Bellaterra, Spain} \email{tommaso@iiia.csic.es}
\date{}

\maketitle

\begin{abstract}
This  paper builds on a recent article  co-authored by the present author, H. Hosni and F. Montagna. It is meant to contribute to the logical foundations of probability theory on many-valued events and, specifically, to a deeper understanding of the notion of strict coherence. In particular, we will make use of geometrical, measure-theoretical and logical methods to provide three characterizations of strict coherence on formulas of infinite-valued \luk\ logic. 
\vspace{.1cm}

\end{abstract}

\section{Introduction and motivation.}\label{sec:intro}
In a collection of seminal contributions starting with \cite{deF0} and culminating in \cite{deF2}, de Finetti grounded subjective probability theory on an ideal betting game between two players, a bookmaker and a gambler, who  wager money on the occurrence 
of certain events $e_1,\ldots, e_k$. For each event $e_i$, gambler's payoffs are $1$ in case $e_i$ occurs, and $0$ otherwise.  The {\em probability} of an event $e_i$ is defined, by de Finetti, as the {\em fair selling price} fixed by the bookmaker for it. 

Conforming to a standard notation, bookmaker's prices for the events $e_1,\ldots, e_k$ will be  referred to as  {\em betting odds}  and an assignment $\beta:\{e_1,\ldots, e_k\}\to[0,1]$ of betting odds $\beta(e_i)=\beta_i$  will be called a {\em book}. 

De Finetti had no particular inclination towards identifying events in a precise logical ground \cite{FGH}. However, in order for his main result to be stated in precise mathematical terms, they will be understood, for the moment, as elements of a finitely generated free boolean algebra and hence coded by boolean formulas. Now, de Finetti's result reads as follows: 
let us fix  finitely many events $e_1,\ldots, e_k$   
and a book $\beta$ on them. A gambler must
choose {\em stakes} $\sigma_1,\ldots, \sigma_k\in \mathbb{R}$, one for each event, and pay to the bookmaker the amount $\sigma_i\cdot \beta_i$ for each $e_i$. 
When a (classical propositional) valuation $w$ determines $e_i$, the gambler gains $\sigma_i$ if $w(e_i) = 1$ and $0$ otherwise. The book $\beta$ is said to be {\em coherent} if there is no choice of stakes $\sigma_1,\ldots, \sigma_k\in \mathbb{R}$ such that for every valuation $w$
\begin{equation}\label{eq:balance}
\sum_{i=1}^k\sigma_i\cdot\beta_i-\sum_{i=1}^k\sigma_i\cdot w(e_i)=\sum_{i=1}^k \sigma_i(\beta_i-w(e_i))<0.
\end{equation}
The left hand side of (\ref{eq:balance}) captures the bookmaker's payoff, or {\em balance}, {\em relative to
the book $\beta$ under the valuation $w$}. 

 Note that a stake $\sigma_i$ may be negative. Following tradition, money transfers are so oriented that ``positive'' means ``gambler-to-bookmaker''. Therefore, if $\sigma_i<0$,  the bookmaker is forced to swap his role with the gambler: he has to pay $-\sigma_i\cdot\beta(e_i)$ to the gambler in hopes of winning $-\sigma_i$ in case $e_i$ occurs. 

De Finetti's Dutch-Book theorem characterizes coherent books as follows: a book $\beta$ on  events $e_1,\ldots, e_k$ pertaining to a boolean algebra ${\bf A}$ is coherent iff it extends to a finitely additive probability $P$ of ${\bf A}$, \cite{deF0}.

Along with the assumptions which regulate  de Finetti's coherence criterion,
condition (\ref{eq:balance}) above effectively forces the bookmaker to set
\emph{fair prices} for gambling on events $e_{1},\ldots ,e_{k}$. In other words,  upon regarding each event $e_i$ as a $\{0,1\}$-valued random variable, 
de Finetti's Dutch-Book theorem amounts to saying
that  coherent assessments are those with null
expectation. For, if the bookmaker publishes a book with positive
expectation (for him)  a logically infallible gambler will choose negative stakes
and inflict a {\em sure loss} on him, that is to say, a sure financial loss whatever the outcome of events.

Although coherence guards the bookmaker against the possibility of sure loss,  at the same time it may bar him from making a profit. To illustrate the idea, consider an event $e$ which is neither noncontradictory nor sure and the coherent book $\beta(e)=0$. 
If the gambler bets $1$ on $e$, then her balance is as follows: she pays $1 \cdot 0 = 0$ and gets back $0$ if $w(e) = 0$ and $1$ if $w(e) = 1$. Hence, the bookmaker never wins and possibly loses.

This rather odd feature of coherence was questioned in the mid 1950's first by Shimony \cite{Shi}  and then by Kemeny \cite{Kem}. These authors studied a refinement of de Finetti's coherence that nowadays goes under the name of {\em strict coherence} (see \cite{FHM18}). Intuitively, a choice of prices is strictly coherent if every possibility of loss, for the bookmaker,  is paired by a possibility of gain.
Precisely, a book $\beta$ is {\em strictly coherent} if, for each choice of stakes $\sigma_1,\ldots, \sigma_k\in \mathbb{R}$,  the existence of a valuation $w$ such that $\sum_{i=1}^k\sigma_i(\beta_i-w(e_i))<0$ implies the existence of another one $w'$ for which $\sum_{i=1}^k\sigma_i(\beta_i-w'(e_i))>0$.

Interest in the condition of strict coherence was prompted by Carnap's analysis of what he called ``regular'' probability functions in \cite{Carnap} (see also \cite[Chapter 10]{Paris2}) and which we will term as {\em Carnap probabilities}.  Those functions arise  from the  axiomatization of finitely additive probabilities by strengthening the usual normalization axiom in the right-to-left direction: $1$ (respectively, $0$) is assigned only to tautologies (respectively, contradictions).  In other words, a  probability function  $P$ is Carnap, if it is normalized, finitely additive and it satisfies $P(e)\neq 0$ for every noncontradictory event $e$.
%
In \cite{FHM18} the authors characterized strictly coherent books on boolean events  
in terms of their extendability to Carnap probabilities\footnote{Carnap probabilities are the same as Carnap-regular probabilities of \cite{FHM18}. In the present paper we adopt this simplified notation in order to avoid any misleading interpretation of the adjective ``regular'' which has indeed different meanings if referred to probability functions or to Borel measures which will be discussed in Section \ref{sec:statesFaith}.}.

Several authors proposed 
generalizations of de Finetti's coherence criterion and his Dutch-Book theorem to events not pertaining to boolean logic. Paris in \cite{Paris}  extended the classical Dutch-Book theorem to several non-classical propositional logics including the  modal logics K, T, S4, S5 and certain paraconsistent logics as well. 
In \cite{Weat}, Weatherson considered the case of events pertaining to intuitionistic logic and in \cite{MuBookmaking}, Mundici extended de Finetti's criterion to the case of infinite-valued \luk\ logic and MV-algebras \cite{CDM,Mu12}.  

In the MV-algebraic realm valuations are $[0,1]$-valued and hence they correspond to homomorphisms into the {\em standard MV-algebra} defined on the unit interval $[0,1]$. De Finetti's coherence criterion immediately translates to the MV-setting with no extra conditions and the main result of \cite{MuBookmaking} (see also \cite{KuMu}) is a de Finetti-like theorem which characterizes coherent books on {\em \luk\ events} as those which are extendible to {\em states}, i.e., $[0,1]$-valued normalized and finitely additive maps of an MV-algebra.

From the perspective of reasoning about uncertainty, the interest in \luk\ events is twofold: on the one hand these events capture properties of the world which are better described as {\em gradual} rather than {\em yes-or-no}; on the other hand, they also mimic bounded random variables. Indeed, any  \luk\ event $e$ 
may be regarded as a $[0,1]$-valued continuous function $f_e$ on a compact Hausdorff space 
(see \cite[Theorem 9.1.5]{CDM} and Section \ref{sec:introAlg}) and any state of $e$ coincides with the expected value of $f_e$ (\cite{KroupaInt, Panti}, \cite[Remark 2.8]{FHL} and Section \ref{sec:statesFaith}).
Therefore, up to renormalization, Mundici's generalization of de Finetti's theorem  \cite[Theorem 2.1]{MuBookmaking} implies the Dutch-Book theorem for books on bounded continuous random variables.

For events pertaining to the restricted class of {\em finite-dimensional} MV-algebras, in \cite[Theorem 6.4]{FHM18} the authors proved a de Finetti-like theorem for strictly coherent books in terms of their extendability to {\em faithful} states, i.e., states which satisfy $s(a)\neq0$ for all $a\neq0$. Nevertheless, extending \cite[Theorem 6.4]{FHM18} to more general classes of MV-algebras is  delicate because, as a consequence of seminal results by Mundici \cite[Proposition 3.2]{Mu}, Kelley \cite{Kelley}, and Gaifman \cite{Gaifman}, an MV-algebra may not have a faithful state.

In this paper we will investigate strictly coherent books on \luk\ events, i.e.,  elements of a finitely generated free MV-algebra. Our results  sensibly extend the results of \cite{FHM18}. In particular, we will provide three characterizations of strict coherence  by adopting  geometrical, measure-theoretical and logical methods. In more details:
\vspace{.2cm}

\noindent {\em Geometrical approach}: the functional representation of $n$-generated free MV-algebras in terms of $n$-variable, piecewise-linear  continuous functions (see \cite[Theorem 9.1.5]{CDM} and \cite{McN}) implies that the set of all coherent books on a finite set $\Phi$ of \luk\ events forms a convex polyhedron $\mathscr{D}_\Phi$ of $\mathbb{R}^k$. The main result of Section \ref{sec:geoStrict} shows that strictly coherent books on $\Phi$ form a subset of $\mathbb{R}^k$ which coincides with the relative interior of $\mathscr{D}_\Phi$. 
\vspace{.2cm}

\noindent {\em Measure-theoretical approach}: faithful states are the  MV-algebraic analog of   Carnap probabilities on a boolean algebra.  In Section \ref{sec:statesFaith} we will first give an integral representation theorem for faithful states on finitely generated free MV-algebras and then we will  characterize strictly coherent books on \luk\ events as those which extend to a faithful state.
\vspace{.2cm}

\noindent {\em Logical approach}:  the  relation among free MV-algebras, rational polyhedra and deducibility in propositional \luk\ logic, will enable us to characterize the notions of coherence and strict coherence  within propositional \luk\ logic (see Section \ref{sec:log}). In our opinion this result is interesting because it shows that propositional \luk\ logic is capable to capture foundational aspects of probability theory on infinite-valued events.
 \vspace{.2cm}
 
In the next section we will introduce necessary preliminaries about MV-algebras and rational polyhedra.

\section{Preliminaries.}\label{sec:introAlg}

The algebraic framework of this paper is that of MV-algebras (see \cite{CDM,Mu12}), i.e., 
the Lindenbaum algebras of \luk\ infinite-valued logic \cite[Definition 4.3.1]{CDM}. A typical example of an MV-algebra is the {\em standard algebra} $[0,1]_{MV}=([0,1], \oplus,  \neg, 0)$ where $x\oplus y=\min\{1, x+y\}$ and $\neg x=1-x$. Further operations, together with their standard interpretation, are defined in $[0,1]_{MV}$ as follows: $x\odot y=\neg(\neg x\oplus\neg y)=\max\{0,x+y-1\}$, $x\to y=\neg x\oplus y=\min\{1, 1-x+y\}$,  $x \wedge y=x\odot(x\to y)=\min\{x,y\}$, $x\vee y=\neg(\neg x\wedge\neg y)=\max\{x,y\}$, $1=\neg 0$. This structure generates the class of MV-algebras both as a variety and as quasi-variety \cite{Chang}.

 Another relevant example of an MV-algebra is given by  the free $n$-generated MV-algebra $\free_n$.  
By a standard universal algebraic argument, $\free_n$ is  the MV-algebra of  functions  $f:[0,1]^n\to[0,1]$ generated by the projection maps \cite[Proposition 3.1.4]{CDM} and whose operations $\odot, \oplus, \to, \wedge, \vee$ and $\neg$  are defined via the pointwise application of those in $[0,1]_{MV}$. By McNaughton theorem, up to isomorphism, $\free_n$ coincides with the MV-algebra of {\em $n$-variable McNaughton functions}: maps  from $[0,1]^n$ to $[0,1]$ which are continuous, piecewise linear, with finitely many pieces, and such that each piece has integer coefficients (cf. \cite[Theorem 9.1.5]{CDM} and \cite{McN}). 
For each $f\in \free_n$, the {\em oneset} of $f$ is $\{x\in [0,1]^n\mid f(x)=1\}$ and the {\em zeroset} of $f$ is $\{x\in [0,1]^n\mid f(x)=0\}$.

The free $n$-generated MV-algebra is, up to isomorphism, the Lindenbaum algebra of \luk\ logic $\L$ in a language with $n$ propositional variables and 
 $[0,1]$-valuations of $\L$ are exactly the homomorphisms of $\free_n$ to $[0,1]_{MV}$. Furthermore, every $x\in [0,1]^n$ determines the homomorphism $h_x:f\in \free_n\mapsto f(x)\in [0,1]_{MV}$.

\begin{prop}[{\cite[Lemma 3.1]{MuBookmaking}}]\label{prop1}
For each finite $n$, homomorphisms of $\free_n$ to $[0,1]_{MV}$, $[0,1]$-valued valuations of \luk\ logic on $n$ variables and points of the $n$-cube $[0,1]^n$ are in one-one correspondence.
\end{prop}
For every closed subset $C$ of $[0,1]^n$, let $I_C$ be the subset of $\free_n$ of those functions whose zeroset  contains  $C$. Then, 
$I_C$ is an {\em ideal} of $\free_n$ and the quotient $\free_n/I_C$ is the MV-algebra whose universe coincides with the set 
given by the restrictions to $C$ of the functions of $\free_n$ (see \cite[Proposition 3.4.5]{CDM}). In particular,  when $C$ has $k$ elements, the quotient MV-algebra $\free_n/I_C$ is isomorphic to the product algebra $[0,1]_{MV}^k$, \cite{CiMa12}.
 The finite powers of $[0,1]_{MV}$ --- called {\em locally weakly finite} MV-algebras in \cite{CiMa12}--- are called,  in this paper,  {\em finite-dimensional}.

\subsection{Rational polyhedra, regular complexes and McNaughton functions.}
In this section we will prepare the necessary results concerning rational and regular complexes (see \cite{Ewald}) and their relation with 
finitely generated free MV-algebras. 
We invite the reader to consult \cite{KroupaPoly,Mu12,Mu11} for  background.

Let $k=1,2,\ldots$. By a  {\em (rational) convex polyhedron} (or {\em (rational) polytope}) of $\mathbb{R}^k$ we mean the convex hull of finitely many points of $\mathbb{R}^k$ ($\mathbb{Q}^k$ respectively); a (rational) polyhedron is  a finite union of (rational) convex polyhedra. Given any  polytope $\mathscr{P}$, we respectively denote by $\ext\; \mathscr{P}$, $\relint\; \mathscr{P}$, $\relb\; \mathscr{P}$ the set of its extremal points, its relative interior and its relative boundary. Since each polytope $\mathscr{P}$ is closed, $\mathscr{P}=\relint\;\mathscr{P}\cup\relb\; \mathscr{P}$. Further, for all vectors $x, y\in \mathbb{R}^k$, we denote $x\cdot y$ their scalar product and by $|x|$ the norm of $x$.

\begin{lem}\label{lemma:conv1}
For each polytope $\mathscr{P}$ of $\mathbb{R}^k$, the following conditions hold:
\begin{enumerate}

\item For every $e\in \relb\;\mathscr{P}$, there exists $\sigma\in \mathbb{R}^k$ such that, for all $\gamma\in \mathscr{P}$, $\sigma\cdot e\leq \sigma\cdot \gamma$;

\item Let $\beta\in \relint\;\mathscr{P}$. Then, there exists $\sigma\in \mathbb{R}^k$ such that the sets $\mathscr{P}_\sigma^+=\{\gamma\in \mathscr{P}\mid \gamma\cdot \sigma<\beta\cdot\sigma\}$ and $\mathscr{P}_\sigma^-=\{\gamma\in \mathscr{P}\mid \gamma\cdot \sigma>\beta\cdot\sigma\}$ are nonempty;
\item Let $\beta\in \relint\;\mathscr{P}$. Then there exist $\sigma\in \mathbb{R}^k$, $e_1,e_2\in \ext\;\mathscr{P}$ such that $e_1\in \mathscr{P}_\sigma^+$ and $e_2\in \mathscr{P}_\sigma^-$;
\item $\gamma\in \relint\;\mathscr{P}$ iff there exists a map $\lambda:\ext\;\mathscr{P}\to[0,1]$ such that $\sum_{e\in \ext\;\mathscr{P}}\lambda(e)=1$, $\lambda(e)>0$ for all $e\in \ext\;\mathscr{P}$ and $\gamma=\sum_{e\in \ext\;\mathscr{P}}\lambda(e)\cdot e$. 
\end{enumerate}
\end{lem}
Before proving the lemma, recall that any hyperplane $H$ of $\mathbb{R}^k$ separates the space in two half spaces denoted  $H^+$ and $H^-$. 
The above claims (2) and (3) state that, if $\beta$ is a point in the relative interior of a polytope $\mathscr{P}$, then  there exists a hyperplane $H$ passing through $\beta$ such that, respectively: both $H^+\cap \mathscr{P}$ and $H^-\cap \mathscr{P}$ are nonempty; each $H^+\cap \mathscr{P}$ and $H^-\cap \mathscr{P}$  contains an extremal point of $\mathscr{P}$.
\begin{proof}
(1) is the well-known supporting hyperplane theorem, see \cite[Theorem 14]{MS71}.
\vspace{.2cm}

\noindent(2) Let $\beta\in \relint\; \mathscr{P}$. Let $\Sigma$ be a sphere of radius $r$ and centered at $\beta$ and  contained in $\relint\;\mathscr{P}$. The existence of $\Sigma$ is ensured by definition of relative interior \cite[Chapter I, Definition 1.8]{Ewald}. Let $\sigma$ be a vector of origin $\beta$. Suppose $\sigma$ is not orthogonal to the affine hull of $\mathscr{P}$ and also $0<|\sigma|<r$. Trivially, $(\sigma-\beta)\cdot \sigma<\beta\cdot \sigma<(\sigma+\beta)\cdot\sigma$. Upon noting that $\sigma-\beta, \sigma+\beta\in \Sigma$, our claim is settled. 
\vspace{.2cm}

\noindent
(3)  By way of contradiction, assume that for no $e\in \ext\;\mathscr{P}$, $e\cdot \sigma<\beta\cdot\sigma$. Equivalently, for all $e\in \ext\;\mathscr{P}$,
\begin{equation}\label{eqextr}
e\cdot \sigma\geq\beta\cdot\sigma.
\end{equation}
Since 
$\mathscr{P}^+$  is nonempty, in view of (2), let $\tau\in \mathscr{P}\cap \mathscr{P}^+$, i.e.,
\begin{equation}\label{eqTauP}
\tau\cdot\sigma<\beta\cdot\sigma.
\end{equation}
If $\tau\in \ext\;\mathscr{P}$ the claim is settled. Assume that $\tau\not\in \ext\;\mathscr{P}=\{e_1,\ldots,e_l\}$. Then there are $\lambda_1,\ldots,\lambda_l\in[0,1]$ such that $\sum_i\lambda_i=1$ and $\tau=\sum_{i}\lambda_i\cdot e_i$. From \eqref{eqextr} it follows that $e_i\cdot \sigma\geq \beta\cdot \sigma$, and hence $\sum\lambda_i e_i\cdot \sigma\geq \beta\cdot \sigma$, that is, $\tau\cdot\sigma\geq \beta\geq \sigma$, which contradicts \eqref{eqTauP}.

\vspace{.2cm}

\noindent(4) See \cite[Lemma 6.1 (1)]{FHM18}.
\end{proof}

Let $k=1,2,\ldots$ and let $x=\langle n_1/d_1, \ldots, n_k/d_k\rangle$ be a rational vector in $\mathbb{R}^k$ with $n_i$ and $d_i$ relatively prime for all $i=1,\ldots,k$. Denote by ${\rm den}(x)$  the least common multiple of $d_1,\ldots, d_k$. The {\em homogeneous correspondent} of $x$ is the vector 
$$
\left\langle \frac{n_1}{d_1}\cdot {\rm den}(x),\ldots, \frac{n_k}{d_k}\cdot {\rm den}(x), {\rm den}(x) \right\rangle\in \mathbb{Z}^{k+1}.
$$
 Let $k=1,2,\ldots$ and $m=0,1,\ldots, k$. A rational $m$-simplex $\conv(x_1,\ldots, x_m)\subseteq \mathbb{R}^k$ is said to be {\em regular}
 if the set of the homogeneous correspondents of $x_1, \ldots, x_m$ is part of a basis of the abelian group $\mathbb{Z}^{m+1}$ \cite[Chapter V, Definition 1.10]{Ewald}. 
 A {\em regular complex} $\Delta$ is a simplicial complex all of whose simplexes are regular\footnote{Recall that  a {\em simplicial complex} $\Delta$ is a nonempty finite set of simplexes  such that: the face of each simplex in $\Delta$ belongs to $\Delta$ and for each pair of simplexes $T_1, T_2\in \Delta$ their intersection is either empty, or it coincides with a common face of $T_1$ and $T_2$.}. 
  Unless otherwise specified, all regular simplexes in this paper are over
  the $n$-cube $[0,1]^n$ (i.e., their faces constitute a {\em unimodular triangulations of $[0,1]^n$}, in the terminology of \cite{MuBookmaking}). Thus, we will  say that $\Delta$ is a {\em regular complex of $[0,1]^n$} without danger of confusion.
  From the regularity of $\Delta$ it follows that $\Delta$ is rational.
   We will denote by $V(\Delta)$ the finite set of rational vertices of $\Delta$, i.e., the union of the set of the vertices of the simplexes in $\Delta$.

Let $\Phi$ be a finite subset of the free $n$-generated MV-algebra $\free_n$. Up to isomorphism, we can (and we will, throughout this paper) think of $\Phi$ as  a finite set of $n$-variable McNaughton functions. Further, if not otherwise specified, we will assume that $\Phi$ has $k$ elements, denoted $f_1,\ldots, f_k$. Following \cite[\S3]{MuBookmaking}, for any $\Phi$, there exists a regular complex $\Delta$ of  $[0,1]^n$ which {\em linearizes} $\Phi$ in the sense that 
each $f_i$ is linear over each simplex of $\Delta$.

\begin{example}\label{example-triang1}
Let us fix $n=2$ and 
consider $\Phi=\{x, y, x\oplus y\}$. Consider  the regular complexes $\Delta_1$ and $\Delta_2$ of Figure \ref{figure:triangulation1} and  whose vertices are $v^1_1=\langle0,0\rangle$, $v^1_2=\langle0,1\rangle$, $v^1_3=\langle1,0\rangle$, $v^1_4=\langle 1,1\rangle$, $v^1_5=\langle 1/2,1/2\rangle$ for $\Delta_1$ and $v^2_1=\langle0,0\rangle$, $v^2_2=\langle0,1\rangle$, $v^2_3=\langle1,0\rangle$, $v^2_4=\langle 1,1\rangle$, $v^2_5=\langle 1/2,1/2\rangle$, $v_6^2=\langle 1/3, 1/3\rangle$, $v_7^2=\langle1/2, 0\rangle$, $v_8^2=\langle 0, 1/2\rangle$ for $\Delta_2$. Then $\Delta_1$ is union of four maximal simplexes, while $\Delta_2$ is union of eight simplexes, see Figure \ref{figure:triangulation1}

Both $\Delta_1$ and $\Delta_2$ linearize $\Phi$. Indeed, for each regular simplex $T$ of $\Delta_1$ and each simplex $T'$ of $\Delta_2$, the restriction of each function $f\in \Phi$ to $T$ and $T'$ is  linear.

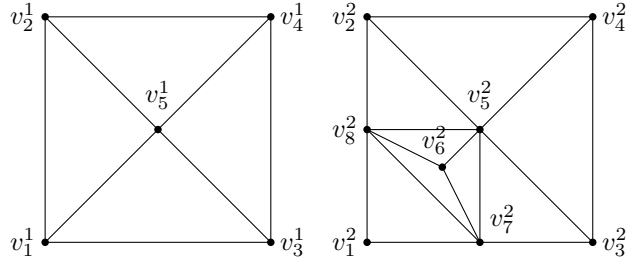
\begin{figure}[h!]
\begin{center}
   \begin{tikzpicture}[scale=3] 
 \coordinate[label=left: {$v_1^1$}] (A) at (0,0) ; 
   \coordinate[label=right: {$v_3^1$}] (B) at (1,0);
       \coordinate[label=right: {$v_4^1$}]  (C) at (1,1);
   \coordinate[label=left: {$v_2^1$}] (D) at (0,1) ; 
       \coordinate[label=above: {$v_5^1$}] (E) at (0.5,0.55) ; 

\node [fill=black,circle,inner sep=1pt,label=-45:{}]  at (1/2,1/2){};
\node [fill=black,circle,inner sep=1pt,label=-45:{}]  at (0,0){};
\node [fill=black,circle,inner sep=1pt,label=-45:{}]  at (0,1){};
\node [fill=black,circle,inner sep=1pt,label=-45:{}]  at (1,0){};
\node [fill=black,circle,inner sep=1pt,label=-45:{}]  at (1,1){};

  \draw (A) -- (B); 
    \draw  (A) -- (C);
      \draw  (D) -- (B);  
  \draw  (A) -- (D);
  \draw  (B) -- (C);
  \draw  (D) -- (C);
  \end{tikzpicture}
   \begin{tikzpicture}[scale=3] 
 \coordinate[label=left: {$v_1^2$}] (A) at (0,0) ; 
   \coordinate[label=right: {$v_3^2$}] (B) at (1,0);
       \coordinate[label=right: {$v_4^2$}]  (C) at (1,1);
   \coordinate[label=left: {$v_2^2$}] (D) at (0,1) ; 
       \coordinate[label=above: {$v_5^2$}] (E) at (0.5,0.55) ; 
             \coordinate[label=above: {$v_6^2$}] (F) at (0.3,0.35) ; 
              \coordinate[label=above: {$v_7^2$}] (G) at (0.6,0) ; 
                 \coordinate[label=left: {$v_8^2$}] (H) at (0,0.5) ; 

\node [fill=black,circle,inner sep=1pt,label=-45:{}]  at (1/2,1/2){};
\node [fill=black,circle,inner sep=1pt,label=-45:{}]  at (0,0){};
\node [fill=black,circle,inner sep=1pt,label=-45:{}]  at (0,1){};
\node [fill=black,circle,inner sep=1pt,label=-45:{}]  at (1,0){};
\node [fill=black,circle,inner sep=1pt,label=-45:{}]  at (1,1){};
\node [fill=black,circle,inner sep=1pt,label=-45:{}]  at (0,1/2){};
\node [fill=black,circle,inner sep=1pt,label=-45:{}]  at (1/2,0){};
\node [fill=black,circle,inner sep=1pt,label=-45:{}]  at (1/3,1/3){};

  \draw (A) -- (B); 
   \draw(1/3,1/3)--(C);
   \draw(1/2,0)--(0,1/2);
   \draw(1/2,0)--(1/3,1/3);
   \draw(1/3,1/3)--(0,1/2);
      \draw  (D) -- (B);  
  \draw  (A) -- (D);
  \draw  (B) -- (C);
  \draw  (D) -- (C);
  \draw(1/2,0)--(1/2,1/2);
  \draw(1/2,1/2)--(H);
  \end{tikzpicture}
  \end{center}
  \caption{{\small The two regular complexes $\Delta_1$ (on the left) and $\Delta_2$ (on the right) of the square $[0,1]^2$. Every function  $x, y$ and $x\oplus y$ is linear over each  simplex of $\Delta_1$ and $\Delta_2$.}}\label{figure:triangulation1}
\end{figure}

\hfill$\Box$
\end{example}

Let $\Delta$ be a regular complex of $[0,1]^n$, and let $v_i$ be one of its vertices. The {\em normalized Schauder hat at $v_i$ (over $\Delta$)} is the uniquely determined continuous   function $\hat{h}_i:[0,1]^n\to[0,1]$ which  is linear over each simplex of $\Delta$ and which attains the value $1$ at $v_i$ and $0$ at all other vertices of $\Delta$. The regularity of $\Delta$ ensures that each linear piece of each $\hat{h}_i$ has  integer coefficients and hence $\hat{h}_i\in \free_n$. By  definition of normalized Schauder hat, $\Delta$ linearizes each $\hat{h}_i$. Further, the following result holds:

\begin{lem}\label{lemmaSchauder}
Let $\Phi$ be a finite subset of $\free_n$, let $\Delta$ be a regular complex linearizing  $\Phi$ and let $v_1,\ldots, v_t$ be the vertices of $\Delta$. Then:
\begin{enumerate}
\item For each $i\neq j$, $\hat{h}_i\odot\hat{h}_j=0$;
\item $\bigoplus_{i=1}^t \hat{h}_i=1$;
\item For each $f\in \Phi$, $f=\bigoplus_{i=1}^t f(v_t)\cdot \hat{h}_{t}$;
\item Let $v_{i_1}, \ldots, v_{i_l}\in V(\Delta)$, and let $\mathscr{P}=\conv(v_{i_1},\ldots, v_{i_l})$. Then the function $p=\bigoplus_{v_j\in\Delta\cap\mathscr{P}}\hat{h}_j$ is a McNaughton function whose oneset is $\mathscr{P}$.
\end{enumerate}
\end{lem}
\begin{proof}
(1) (2) and (3) have been proved in \cite[Lemma 3.4 (ii), (iii), (iv) and (v)]{MuBookmaking}. Let next prove (4). First of all, $p=\bigoplus_{v\in\Delta\cap\mathscr{P}}\hat{h}_v$ is  a McNaughton function by definition. Further, for every vertex $v\in \Delta\cap\mathscr{P}$, $p(v)=1$ by definition of normalized Schauder hat. If $x\in \mathscr{P}\setminus V(\Delta)$, let $\Sigma$ be a simplex of $\Delta$ which contains $x$. The claim follows, since  each $\hat{h}_j$ is linear on $\Sigma$.
\end{proof}

\section{A geometric characterization of strict coherence.}\label{sec:geoStrict}
By Proposition \ref{prop1}, if $\Phi=\{f_1,\ldots, f_k\}$ is a finite subset of $\free_n$ and $\beta$ a book on $\Phi$, we can rephrase the definitions of coherence and strict coherence for $\beta$ as follows: 
\begin{enumerate}
\item $\beta$ is coherent if for every $\sigma\in \mathbb{R}^k$, there exists $x\in [0,1]^n$ such that $\sigma \cdot \langle f_1(x),\ldots, f_k(x) \rangle\geq 0$.
\item $\beta$ is {\em strictly coherent} if for every $\sigma\in \mathbb{R}^k$, the existence of $x\in [0,1]^n$ such that $\sigma \cdot \langle f_1(x),\ldots, f_k(x) \rangle<0$, implies the existence of another $x'\in [0,1]^n$ such that 
$\sigma \cdot \langle f_1(x'),\ldots, f_k(x') \rangle>0$.
\end{enumerate}
Notice that a book $\beta$ is coherent and not strictly coherent iff for any vector $\sigma\in \mathbb{R}^k$, one has that for all $x\in [0,1]^n$, $\sigma \cdot \langle f_1(x),\ldots, f_k(x) \rangle\leq0$ and for some $x'\in [0,1]^n$, $\sigma \cdot \langle f_1(x'),\ldots, f_k(x') \rangle=0$.

Throughout we will adopt the following notation:
$$
\mathscr{D}_\Phi=\{\beta:\Phi\to[0,1]\mid \beta \mbox{ is coherent}\}.
$$

For any  $X\subseteq[0,1]^n$, $\mathscr{C}_\Phi(X)$ will denote the topological closure of the convex hull of all points of $\mathbb{R}^k$ of the form $\langle f_1(x),\ldots, f_k(x)\rangle$ for $\varphi_i\in \Phi$ and $x\in X$. In symbols,
$$
\mathscr{C}_\Phi(X)=\cl\;\conv\{\langle f_1(x),\ldots, f_k(x)\rangle\mid \varphi_i\in \Phi, x\in X\}.
$$
Whenever $X$ is finite, $\mathscr{C}_\Phi(X)=\conv\{\langle f_1(x),\ldots, f_k(x)\rangle\mid \varphi_i\in \Phi, x\in X\}$, which  is a convex polytope. For the sake of readability, we will  write $\mathscr{C}_\Phi$ instead of $\mathscr{C}_\Phi([0,1]^n)$.

For every finite $\Phi$, Mundici's extension of de Finetti's theorem to \luk\ logic (see \cite[Theorem 2.1]{MuBookmaking}) shows that $\mathscr{D}_\Phi=\mathscr{C}_\Phi$. 

\begin{lem}{\cite[Corollary 5.4]{MuBookmaking}}\label{lemma1}
For any book $\beta:\Phi\to[0,1]$ the following conditions are equivalent:
\begin{enumerate}
\item $\beta$ is coherent;
\item There exists a finite $X\subset[0,1]^n$ such that, for each $f_i\in \Phi$,  $\beta(f_i)\in \mathscr{C}_\Phi(X)$;
\item There exists a finite $X\subset[0,1]^n$ with $|X|\leq n+1$ such that, for each $f_i\in \Phi$,  $\beta(f_i)\in \mathscr{C}_\Phi(X)$;
\item For each regular complex $\Delta$ of $[0,1]^n$ linearizing  $\Phi$, $\beta(f_i)\in \mathscr{C}_\Phi(V(\Delta))$.
\end{enumerate}
\end{lem}

The next corollary is an immediate consequence of Lemma \ref{lemma1}.

\begin{cor}\label{corollary1}
Let $\Phi$ be a finite set of \luk\ formulas. Then, for every regular complex $\Delta$ which linearizes $\Phi$,
$$
\mathscr{D}_\Phi=\mathscr{C}_\Phi(V(\Delta)).
$$
\end{cor}
Thus, the coherence of a  book $\beta$ on $\Phi$ does not depend on the particular regular complex $\Delta$ chosen to linearize $\Phi$. Moreover, since for each $\Delta$, $V(\Delta)$ is finite,  $\mathscr{C}_\Phi(V(\Delta))$ is a polytope coinciding with $\mathscr{D}_\Phi$. Therefore, 
by the Krein-Milman theorem \cite[Theorem 1.2]{Ewald}, $\mathscr{C}_\Phi(V(\Delta))$ is the convex hull of the set of extremal points of $\mathscr{D}_\Phi$, i.e., for every 
$\Delta$, 
$$
\mathscr{C}_\Phi(V(\Delta))=\conv\; \ext\; \mathscr{D}_\Phi.
$$
The following example clarifies the claim made in Corollary \ref{corollary1}.

\begin{example}\label{rem:dependance} 
Let $\Phi=\{x, y, x\oplus y\}$ together with the regular complexes $\Delta_1$ and $\Delta_2$ of Example \ref{example-triang1}. 

Each of the five vertices $v^1_i$ of $\Delta_1$  determines a point 
$$
p_i=\langle f_1(v^1_i), f_2(v^1_i), f_3(v^1_i)\rangle\in \mathbb{R}^3
$$ 
(where  $f_1(x,y)=x, f_2(x,y)=y$ and $f_3(x, y)=x\oplus y$) and 
$$
\mathscr{D}_\Phi=\mathscr{C}_\Phi(\{v^1_1,\ldots, v^1_5\})=\conv(\{p_1,\ldots, p_5\}).
$$ 
In particular: $p_1=\langle 0,0,0\rangle$, $p_2=\langle 0,1,1\rangle$, $p_3=\langle 1,0,1\rangle$, $p_4=\langle 1,1,1\rangle$ and $p_5=\langle 1/2, 1/2, 1\rangle$. 

Similarly, for $\Delta_2$,
$$
\mathscr{D}_\Phi=\mathscr{C}_\Phi(\{v^2_1,\ldots, v^2_8\})=\conv(\{q_1,\ldots, q_8\}),
$$
where:  $q_1=p_1=\langle 0,0,0\rangle$, $q_2=p_2=\langle 0,1,1\rangle$, $q_3=p_3=\langle 1,0,1\rangle$, $q_4=p_4=\langle 1,1,1\rangle$, $q_5=p_5=\langle 1/2, 1/2, 1\rangle$, $q_6=\langle 1/3, 1/3, 2/3\rangle$, $q_7=\langle 1/2,0,1/2\rangle$ and $q_8=\langle 0, 1/2, 1/2\rangle$.

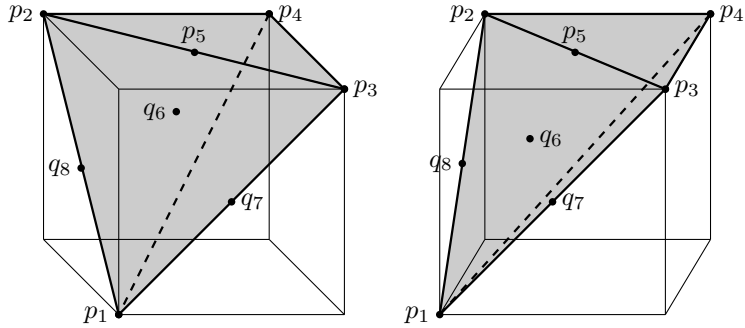
\begin{figure}[h!]
\begin{center}
   \begin{tikzpicture}[scale=3] 
 \coordinate[label=left: {$p_1$}] (A) at (0,0) ; 
  \coordinate (F) at (-1/3,1/3) ; 
   \coordinate(B) at (1,0);
      \coordinate(H) at (2/3,1/3);
       \coordinate[label=right: {$p_3$}]  (C) at (1,1);
              \coordinate[label=right: {$p_4$}]  (I) at (2/3,4/3);

   \coordinate(D) at (0,1) ; 
      \coordinate[label=left: {$p_2$}] (G) at (-1/3,4/3) ; 
         \coordinate[label=above: {$p_5$}] (L) at (.336,1.163) ; 
\node [fill=black,circle,inner sep=1pt,label=-45:{}]  at (0,0){};
\node [fill=black,circle,inner sep=1pt,label=-45:{}]  at (1,1){};
\node [fill=black,circle,inner sep=1pt,label=-45:{}]  at (-1/3,4/3){};
\node [fill=black,circle,inner sep=1pt,label=-45:{}]  at (2/3,4/3){};
\node [fill=black,circle,inner sep=1pt,label=-45:{}]  at (.336,1.163){};

   \filldraw[fill= black!20!white, draw=black]
  (-1/3,4/3)--(2/3,4/3)--(1,1);
    \filldraw[fill= black!20!white, draw=black] 
   (0,0)--(1,1)--(-1/3,4/3);
  
\node [fill=black,circle,inner sep=1pt,label=-45:{}]  at (0,0){};
\node [fill=black,circle,inner sep=1pt,label=-45:{}]  at (1,1){};
\node [fill=black,circle,inner sep=1pt,label=-45:{}]  at (2/3,4/3){};
\node [fill=black,circle,inner sep=1pt,label=-45:{}]  at (.336,1.163){};
         \coordinate[label=above: {$p_5$}] (L) at (.336,1.163) ; 
\draw[line width=0.3mm] (C)--(A);
\draw[line width=0.3mm,dashed](I)--(A);
\draw[line width=0.3mm] (G)--(A);
\draw[line width=0.3mm] (C)--(G);
\draw (A)--(F);
\draw (F)--(G);
\draw [line width=0.3mm] (G)--(I);
\draw[line width=0.3mm]  (I)--(C);
\draw (G)--(D);
\draw (B)--(H);
\draw (I)--(H);
\draw (H)--(F);
  \draw  (A) -- (B); 
  \draw  (A) -- (D);
  \draw  (B) -- (C);
  \draw  (D) -- (C);
   \coordinate[label=left: {$q_6$}] (K)  at (0.255,0.9);
   \node [fill=black,circle,inner sep=1pt,label=-45:{}]  at (0.255,0.9){};
      \coordinate[label=right: {$q_7$}] (Z)  at (0.5,0.5);
   \node [fill=black,circle,inner sep=1pt,label=-45:{}]  at (0.5,0.5){};
         \coordinate[label=left: {$q_8$}] (T)  at (-1/6,0.65);
   \node [fill=black,circle,inner sep=1pt,label=-45:{}]  at (-1/6,0.65){};
  \end{tikzpicture}
     \begin{tikzpicture}[scale=3] 
 \coordinate[label=left: {$p_1$}] (A) at (0,0) ; 
  \coordinate (F) at (0.2,1/3) ; 
   \coordinate(B) at (1,0);
      \coordinate(H) at (1.2,1/3);
       \coordinate[label=right: {$p_3$}]  (C) at (1,1);
              \coordinate[label=right: {$p_4$}]  (I) at (1.2,4/3);
   \coordinate(D) at (0,1) ; 
      \coordinate[label=left: {$p_2$}] (G) at (0.2,4/3) ;

\node [fill=black,circle,inner sep=1pt,label=-45:{}]  at (0,0){};
\node [fill=black,circle,inner sep=1pt,label=-45:{}]  at (1,1){};

\node [fill=black,circle,inner sep=1pt,label=-45:{}]  at (1.2,4/3){};
\node [fill=black,circle,inner sep=1pt,label=-45:{}]  at (.6,1.163){};

   \filldraw[fill= black!20!white, draw=black]
  (0.2,4/3)--(1.2,4/3)--(1,1);
    \filldraw[fill= black!20!white, draw=black] 
   (0,0)--(1,1)--(0.2,4/3);
  
  \node [fill=black,circle,inner sep=1pt,label=-45:{}]  at (1,1){};
\node [fill=black,circle,inner sep=1pt,label=-45:{}]  at (.6,1.163){};
        \node [fill=black,circle,inner sep=1pt,label=-45:{}]  at (0.2,4/3){};
                 \coordinate[label=above: {$p_5$}] (L) at (.6,1.163) ; 
                 \node [fill=black,circle,inner sep=1pt,label=-45:{}]  at (1.2,4/3){};
\draw[line width=0.3mm] (C)--(A);
\draw[line width=0.3mm,dashed](I)--(A);
\draw[line width=0.3mm] (G)--(A);
\draw[line width=0.3mm] (C)--(G);
\draw (A)--(F);
\draw (F)--(G);
\draw [line width=0.3mm] (G)--(I);
\draw[line width=0.3mm]  (I)--(C);
\draw (G)--(D);
\draw (B)--(H);
\draw (I)--(H);
\draw (H)--(F);
  \draw  (A) -- (B); 

  \draw  (A) -- (D);
  \draw  (B) -- (C);
  \draw  (D) -- (C);
    \coordinate[label=right: {$q_6$}] (K)  at (0.4,0.78);
   \node [fill=black,circle,inner sep=1pt,label=-45:{}]  at (0.4,0.78){};
    \coordinate[label=right: {$q_7$}] (Z)  at (0.5,0.5);
   \node [fill=black,circle,inner sep=1pt,label=-45:{}]  at (0.5,0.5){};
       \coordinate[label=left: {$q_8$}] (T)  at (0.1,0.67);
   \node [fill=black,circle,inner sep=1pt,label=-45:{}]  at (0.1,0.67){};
  \end{tikzpicture}
  \end{center}
  \caption{{\small The convex polytope $\mathscr{D}_\Phi$ (in two perspectives) for $\Phi=\{x, y, x\oplus y\}$, and its extremal points $p_1, p_2, p_3$ and $p_4$.}}\label{figure:triangulation2}
\end{figure}

Since both $\Delta_1$ and $\Delta_2$ linearize $\Phi$, 
$$
\mathscr{D}_\Phi=\mathscr{C}_\Phi(V(\Delta_1))=\mathscr{C}_\Phi(V(\Delta_2))
$$ 
(see Figure \ref{figure:triangulation2}) and 
$$
\ext\;\mathscr{C}_\Phi(V(\Delta_1))=\ext\;\mathscr{C}_\Phi(V(\Delta_2))=\ext\;\mathscr{D}_\Phi=\{p_1, p_2, p_3, p_4\}.
$$
\hfill$\Box$
\end{example}

Let us write:
$$
\mathscr{K}_\Phi=\{\beta:\Phi\to[0,1]\mid \beta \mbox{ is strictly coherent}\}.
$$

The following theorem, which is the main result of this section,  provides us with a geometric characterization of strict coherence for books on formulas of \luk\ logic. 
\begin{thm}\label{thm:MainGeo}
Let $\Phi$ be a finite subset of $\free_n$. Then
$$
\mathscr{K}_\Phi=\relint\;\mathscr{D}_\Phi.
$$
\end{thm}
\begin{proof}
Since $\mathscr{D}_\Phi=\mathscr{C}_\Phi$, we will prove the equivalent claim: $\mathscr{K}_\Phi=\relint\;\mathscr{C}_\Phi$.

Trivially, $\mathscr{K}_\Phi\subseteq \mathscr{C}_\Phi$. Let us show that $\mathscr{K}_\Phi\subseteq \relint\;\mathscr{C}_\Phi$. Since $\mathscr{C}_\Phi$ is closed, 
\begin{center}
$\mathscr{C}_\Phi=\relint\; \mathscr{C}_\Phi\cup \relb\; \mathscr{C}_\Phi$ and $\relint\; \mathscr{C}_\Phi\cap \relb\; \mathscr{C}_\Phi=\emptyset$.
\end{center}
Assume (absurdum hypothesis) that $\beta\in \mathscr{K}_\Phi\cap \relb\; \mathscr{C}_\Phi$. By Lemma \ref{lemma:conv1} (1)  
there exists $\sigma\in \mathbb{R}^k$ such that for all $\gamma\in \mathscr{C}_\Phi$, 
$$
\sigma \cdot \beta\leq \sigma\cdot \gamma.
$$ 
Thus, for all $x\in [0,1]^n$, 
$$
\sigma \cdot \beta\leq \sigma\cdot \langle f_1(x),\ldots, f_k(x)\rangle\leq 0.
$$
Therefore, $\beta$ is coherent but not strictly coherent. This contradicts our hypothesis. Thus,  $\mathscr{K}_\Phi\subseteq \relint\;\mathscr{C}_\Phi$.

In order to prove the converse inclusion, assume that $\beta\in \relint\; \mathscr{C}_\Phi$ and let 
$\sigma\in \mathbb{R}^k$ satisfy  
Lemma \ref{lemma:conv1} (2). Then
$$
(\mathscr{C}_\Phi)_\sigma^+=\{\gamma\in \mathscr{C}_\Phi\mid \gamma\cdot \sigma < \beta\cdot \sigma\}
$$
and
$$
(\mathscr{C}_\Phi)_\sigma^-=\{\gamma\in \mathscr{C}_\Phi\mid \gamma\cdot \sigma > \beta\cdot \sigma\},
$$
are nonempty.

Moreover, by Lemma \ref{lemma:conv1} (3), both $(\mathscr{C}_\Phi)_\sigma^+$ and $(\mathscr{C}_\Phi)_\sigma^-$  contain  an extremal point of $\mathscr{C}_\Phi$. Therefore, there are $x, x'\in [0,1]^n$ such that $\langle f_1(x),\ldots, f_k(x)\rangle \cdot \sigma < \beta\cdot \sigma$ and $\langle f_1(x'),\ldots, f_k(x')\rangle \cdot \sigma > \beta\cdot \sigma$, that is, $\beta$ is strictly coherent.
\end{proof}
\begin{cor}\label{cor:decidable}
Let $\Phi$ and $\beta$ be as above. Then the following conditions are equivalent:
\begin{enumerate}
\item $\beta$ is strictly coherent;
\item For each regular complex $\Delta$ which linearizes $\Phi$, there is  $\lambda:V(\Delta)\to[0,1]$ such that $\sum_{v_i\in V(\Delta)}\lambda(v_i)=1$, for all $v_i\in V(\Delta)$, $\lambda(v_i)>0$ and $\beta(f_j)=\sum_{v_i\in V(\Delta)}\lambda(v_i)\cdot f_j(v_i)$;
\item There exists a map $\lambda:\ext\; \mathscr{D}_\Phi\to[0,1]$ such that $\sum_{e\in \ext\; \mathscr{D}_\Phi}\lambda(e)=1$, for all $e\in \ext\; \mathscr{D}_\Phi$, $\lambda(e)>0$ and $\beta(\varphi_i)=\sum_{e\in \ext\; \mathscr{D}_\Phi}\lambda(e)\cdot f_i(e)$.
\end{enumerate}
Further, the set $\mathscr{K}_\Phi^\mathbb{Q}$ of rational-valued strictly coherent books on $\Phi$ is decidable.
\end{cor}
\begin{proof}
The equivalence between (1), (2) and (3)  follows from Theorem \ref{thm:MainGeo}, Corollary \ref{corollary1} and Lemma \ref{lemma:conv1} (4). To conclude the proof we will prove the decidability of the set $\mathscr{K}_\Phi^\mathbb{Q}$. To this purpose, given $\Phi$, the problem of determining a regular complex $\Delta$ which linearizes all McNaughton functions $f_i\in \Phi$,  is computable by a Turing machine (see \cite[Theorem 7.1, Claim 3]{MuBookmaking}).
Therefore, by (2), $\mathscr{K}_\Phi^\mathbb{Q}$ is decidable iff the following {\em bounded mixed integer programming} problem (see \cite{Hah}) with unknowns $\lambda(v_i)$ for all $v_i\in V(\Delta)$,  has a solution in $[0,1]\cap\mathbb{Q}$:
$$
(S_{\mathscr{K}_\Phi})=
\left\{\begin{array}{l}
\lambda(v_i)>0,\\
\sum_{v_i}\lambda(v_i)=1,\\
\sum_{v_i}\lambda(v_i)\cdot f_\varphi(v_i)=\beta(\varphi).
\end{array}
\right.
$$
 Thus the decidability of $\mathscr{K}_\Phi^\mathbb{Q}$ follows from \cite[Proposition 2]{Hah}.

\end{proof}
\section{Strict coherence, infinite-valued events and faithful states.}\label{sec:statesFaith}
Generalizing de Finetti's theorem, a book on \luk\ events is {\em coherent} iff it can be extended to a state in the sense of the following definition.

\begin{definition}[\cite{Mu}]
A {\em state} of an MV-algebra ${\bf A}$ is a map $s:A\to[0,1]$ satisfying the following conditions:
\begin{itemize}
\item[$(s1)$] Normalization: $s(1)=1$,
\item[$(s2)$] Additivity: $s(a\oplus b)=s(a)+ s(b)$, for all $a,b\in A$ such that $a\odot b=0$.
\end{itemize}
A state $s$ is said to be {\em faithful}  if  $s(a)\neq0$ for all $a\neq0$.
\end{definition}

Kroupa and Panti independently proved that for every state $s$ of an MV-algebra ${\bf A}$ there exists a unique regular Borel, and hence $\sigma$-additive, probability measure $\mu_s$ on the space of maximal ideals with the hull-kernel topology of ${\bf A}$ such that $s$ is the integral with respect to $\mu_s$ (see \cite{KroupaInt}, \cite{Panti} and \cite[\S10]{Mu12}). In particular, for $n$-generated free MV-algebras, the Kroupa-Panti theorem shows that for every state $s$ of $\free_n$ there exists a unique regular Borel probability measure $\mu_s$ on $[0,1]^n$ such that for each $f\in \free_n$,
\begin{equation}\label{eq:KP}
s(f)=\int_{[0,1]^n}f\;{\rm d}\mu_s.
\end{equation}
The correspondence between  states of $\free_n$ and regular Borel probability measures on $[0,1]^n$ is one-one.

 The next result, which to the best of our knowledge is new, represents faithful states of $\free_n$ in a similar manner. Following \cite{Todo}, we say that a regular Borel measure $\mu$ of $[0,1]^n$ is  {\em strictly positive} if for every nonempty open $O\subseteq[0,1]^n$, $\mu(O)>0$.
 
\begin{prop}\label{propFaith}
For any state $s$ of $\free_n$ the following conditions are equivalent:
\begin{enumerate}
\item $s$ is a faithful state;
\item There exists a unique strictly positive, regular probability Borel measure $\mu_s$ such that for every $f\in \free_n$,
$$
s(f)=\int_{[0,1]^n} f\;{\rm d}\mu_s.
$$
\end{enumerate}
The correspondence between faithful states of $\free_n$ and strictly positive, regular probability Borel measures of $[0,1]^n$ is one-one.
\end{prop}
\begin{proof}
For every state $s$ of $\free_n$, let $\mu_s$ be the unique  regular probability Borel measure of $[0,1]^n$ such that for every $f\in \free_n$,
$
s(f)=\int_{[0,1]^n} f\;{\rm d}\mu_s
$
as in the Kroupa-Panti theorem.
\vspace{.2cm}
 
$(1)\Rightarrow(2)$. Assume that $O$ is a nonempty open set in the product topology of $[0,1]^n$, and $\mu_s(O)=0$. Let $K_O$ be any nonempty compact subset of $O$ and assume, without loss of generality, that $K_O$ is a rational polyhedron. By \cite[Corollary 2.10]{Mu12}, 
there exists  $f\in \free_n$ such that $K_O$ is the oneset of $f$. Since $K_O$ is contained in $O$, by \cite[Lemma 2.2 (i)]{KroupaPoly} there exists  $n\in \mathbb{N}$ such that, for all $x\in [0,1]^n$, the $n$-fold $\odot$-product $f^n=f\odot\ldots\odot f$, satisfies $f^n(x)=1$ if $x\in K_O$ and $f^n(x)=0$ for all $x\not\in O$. Therefore, $f^n\neq 0$ and $s(f^n)=\int_{[0,1]^n}f^n\; {\rm d}\mu_s=0$ whence $s$ is not faithful.
\vspace{.2cm}

\noindent $(2)\Rightarrow(1)$. Assume that $s$ is not faithful and in particular, let $f\in \free_n$ be such that $f\neq 0$ and $s(f)=0$. Since $f$ is continuous and not constantly $0$, its support ${\rm supp}(f)=\{x\in [0,1]^n\mid f(x)>0\}$ is nonempty and open. 
Thus,  
$$
0=s(f)=\int_{[0,1]^n}f\;{\rm d}\mu_s=\int_{{\rm supp}(f)}f\;{\rm d}\mu_s,
$$ 
whence $\mu_s({\rm supp}(f))=0$. 
\end{proof}

\begin{remark}\label{remFaith}
From Proposition \ref{propFaith} it follows that if  ${\bf A}=\free_{n}/I_C$ is a finite-dimensional MV-algebra,  there is a one-one correspondence between faithful states of ${\bf A}$,  strictly positive distributions on the  points $c_1,\ldots, c_t$ of $C$, and points in the relative interior of the simplex $\Sigma_\Delta=\left\{\langle\lambda_1,\ldots, \lambda_{t}\rangle\in \mathbb{R}^{t}\mid\sum_{i=1}^t\lambda_i=1\right\}$ (see \cite[Remark 6.3]{FHM18}). 
Every finitely generated free boolean algebra ${\bf A}$ is, in particular, a finite-dimensional MV-algebra and every faithful state $s$ of ${\bf A}$ is a Carnap probability (recall Section \ref{sec:intro}). Therefore, Proposition \ref{propFaith} specializes on boolean events as follows: for every $n=1,2,\ldots$, a finitely additive probability $P$ of the $n$-generated free boolean algebra ${\bf A}$ is Carnap iff there exists a unique strictly positive distribution $\mu_P$ on $\{0,1\}^n$ such that for every $f\in A$,
$
P(f)=\sum_{x\in\{0,1\}^n}\mu_P(x)\cdot f(x). 
$
\end{remark}
 
In \cite{FHM18}, the authors characterized strictly coherent books on finite subsets of a finite-dimensional MV-algebra ${\bf A}$ (recall Section \ref{sec:introAlg}) as those books that can be extended to a faithful state of ${\bf A}$. 
In this section, we will provide two measure-theoretical characterizations of strict coherence for books on $\free_n$. The first one (Theorem \ref{thm:strictFree}) involves states satisfying a {\em local} version of faithfulness which depends both on $\Phi$ and on the fixed regular complex linearizing its functions; the second one (Theorem \ref{main2}) is given in terms of faithful states of $\free_n$.

\begin{definition}
Let $\Phi$ be a finite subset of $\free_n$ and let $\Delta$ be a regular complex which linearizes $\Phi$. Then $s$ is said to be {\em $\Delta$-faithful} provided that $s(\hat{h}_v)>0$ for all $v\in V(\Delta)$.
\end{definition}

The next lemma collects some useful properties of states and $\Delta$-faithful states.
\begin{lem}\label{lemma:statesFinite}
For each regular complex $\Delta$ of $[0,1]^n$ with vertices $V(\Delta)=\{v_1,\ldots, v_t\}$, the following conditions hold:
\begin{enumerate}
\item For each map $\lambda:V(\Delta)\to[0,1]$ such that $\sum_{v\in V(\Delta)}\lambda(v)=1$, the map $s_\lambda:\free_n\to [0,1]$ defined as 
\begin{equation}\label{eq:slambda}
s_\lambda(f)=\sum_{v\in V(\Delta)} f(v)\cdot \lambda(v).
\end{equation}
is a state of $\free_n$.
\item The set of $\Delta$-faithful states of $\free_n$ is in one-one correspondence with the set of faithful states of $\free_n/I_{V(\Delta)}$ and  hence is in one-one correspondence  
with the relative interior of the simplex 
$$
\Sigma_\Delta=\left\{\langle\lambda_1,\ldots, \lambda_t\rangle\in \mathbb{R}^t\mid\sum_{i=1}^t\lambda_i=1\right\}.
$$
\end{enumerate}
\end{lem}
\begin{proof}
(1). Every state of $\free_n$ belongs to the closure of the convex hull of the homomorphisms of $\free_n$ to $[0,1]_{MV}$ (see \cite[Theorem 2.5]{Mu} and \cite[Theorem 4.1.1]{FK15}). Thus the claim follows immediately from Proposition \ref{prop1}

\vspace{.2cm}

\noindent(2). 
The claim easily follows from Proposition \ref{propFaith}, Remark \ref{remFaith} and the  definition of $\Delta$-faithfulness.
\end{proof}

The next theorem yields a characterization of strictly coherent books in terms of $\Delta$-faithful states. 
\begin{thm}\label{thm:strictFree}
Let $\Phi$ be a finite subset of $\free_n$ and let $\beta$ be a book on $\Phi$. Then the following conditions are equivalent:
\begin{enumerate}
\item $\beta$ is strictly coherent;
\item There exists a regular complex $\Delta$ which linearizes $\Phi$ and a $\Delta$-faithful state $s$ which extends $\beta$;
\item For every regular complex $\Delta$ which linearizes $\Phi$, there exists a $\Delta$-faithful state $s$ of $\free_n$ which extends $\beta$. 
\end{enumerate}
\end{thm}
\begin{proof}
$(1)\Rightarrow(3)$. Let $\beta$ be strictly coherent. From Corollary \ref{cor:decidable} (3), for every regular complex $\Delta$ linearizing $\Phi$, there exists a map $\lambda:V(\Delta)\to[0,1]$ such that $\sum_{v\in V(\Delta)}\lambda(v)=1$, $\lambda(v)>0$ for all $v\in V(\Delta)$ and for every $f_j\in \Phi$, 
\begin{equation}\label{eq:beta1}
\beta(f_j)=\sum_{v\in V(\Delta)}f_j(v)\cdot \lambda(v).
\end{equation}
Let $s_\lambda$ be the state of $\free_n$ defined in (\ref{eq:slambda}). 

First of all notice that, directly from (\ref{eq:slambda}) and (\ref{eq:beta1}), $s_\lambda$ extends $\beta$. Furthermore, for every vertex  $v\in V(\Delta)$, the normalized Schauder hat $\hat{h}_v$ takes value $1$ on $v$ and $0$ on any  $v'\neq v$. Thus, $s_\lambda(\hat{h}_v)=\sum_{v'\in V(\Delta)}\hat{h}_v(v')\cdot \lambda(v')=\lambda(v)$ and hence $s_\lambda(\hat{h}_v)>0$. Therefore $s_\lambda$ is a $\Delta$-faithful state of $\free_n$ which extends $\beta$.  
\vspace{.2cm}

\noindent$(3)\Rightarrow(1)$. Now assume that (3) holds and define $\lambda:V(\Delta)\to[0,1]$ by
$$
\lambda(v)=s(\hat{h}_v).
$$
From Lemma \ref{lemmaSchauder},  
$$
\sum_{v\in V(\Delta)}\lambda(v)=\sum_{v\in V(\Delta)}s(\hat{h}_v)=s\left(\bigoplus_{v\in V(\Delta)}\hat{h}_v\right)=s(1)=1. 
$$
Since $s$ is $\Delta$-faithful, $\lambda(v)>0$ for each $v\in V(\Delta)$. Further, for all $f_i\in \Phi$, $\beta(f_i)=s(f_i)=\sum_{v\in V(\Delta)} \lambda_{v}\cdot f_i(v)$. 
Thus, $\beta$ is strictly coherent by Corollary \ref{cor:decidable} ($(2)\Rightarrow(1)$).
\vspace{.2cm}

\noindent Finally, $(3)\Rightarrow(2)$ is trivial and  $(2)\Rightarrow(3)$ follows from $(1)\Leftrightarrow(3)$ above, Corollary \ref{corollary1} and  Theorem \ref{thm:MainGeo}
\end{proof}

 Lemma \ref{lemma:statesFinite} (2) shows that for each regular complex $\Delta$ which linearizes $\Phi$, $\Delta$-faithful states of $\free_n$ are in one-one correspondence with faithful states of $\free_n/I_{V(\Delta)}$. In particular, for  every $\Delta$-faithful  state $s$ of $\free_n$, let $s_\Delta$ be the unique faithful state of $\free_n/I_{V(\Delta)}$ such that: for every $f\in \free_n$, let $f_\Delta$ to be the restriction of $f$ to $V(\Delta)$ and
$$
s_\Delta(f_\Delta)=\sum_{v\in V(\Delta)} s(\hat{h}_v)\cdot f_\Delta(v).
$$
Thus, if $f\in \Phi$,  $s_\Delta(f_\Delta)=s(f)$. We then have:

\begin{cor}
Let $\Phi$ be a finite subset of $\free_n$ and let $\beta$ be a book on $\Phi$. Then the following conditions are equivalent:
\begin{enumerate}
\item $\beta$ is strictly coherent;
\item There exists a regular complex $\Delta$ which linearizes $\Phi$ and a faithful state $s_\Delta$ of $\free_n/I_{V(\Delta)}$ such that, for all $f\in \Phi$, $\beta(f)=s_\Delta(f_\Delta)$;
\item For every regular complex $\Delta$ which linearizes $\Phi$, there exists a faithful state $s_\Delta$ of $\free_n/I_{V(\Delta)}$ such that, for all $f\in \Phi$, $\beta(f)=s_\Delta(f_\Delta)$. 
\end{enumerate}
\end{cor}

The following construction is used in the next result which characterizes strictly coherent books in terms of faithful states: 
let $\beta$ be a strictly coherent book on a finite subset $\Phi$ of $\free_n$, fix an enumeration $g_1,g_2,\ldots$ of $\free_{n}\setminus\{\Phi, 0,1\}$ and consider the following inductive construction:
\begin{itemize} 
\item[($S_1$)] Put $\Phi_1=\Phi\cup\{g_1\}$. 
Each  regular complex $\Delta_1$ linearizing $\Phi_1$ also linearizes $\Phi$. Since $\beta$ is strictly coherent,  Theorem \ref{thm:strictFree} yields a $\Delta_1$-faithful state $s_1$ which extends $\beta$. It follows that the extended book $\beta_1=\beta\cup\{g_1\mapsto s_1(g_1)\}$ is strictly coherent because $s_1$ extends it. Further,  $0<s_1(g_1)<1$.

\item[($S_2$)] Consider $\Phi_2=\Phi_1\cup\{g_2\}$ and fix $\Delta_2$ that linearizes $\Phi_2$ and a $\Delta_2$-faithful state $s_2$ which extends $\beta_1$. Again, $0<s_2(g_2)<1$ and $\beta_2=\beta_1\cup\{g_2\mapsto s_2(g_2)\}$ is strictly coherent by Theorem \ref{thm:strictFree}. Further, $s_2(e)=s_1(e)$ for all $e\in \Phi_1$.

\item[($S_{i+1}$)] At step $i+1$, arguing by induction, construct a regular complex $\Delta_{i+1}$ which linearizes $\Phi_i\cup\{g_{i+1}\}=\Phi\cup\{g_1,\ldots, g_{i+1}\}$, a state $s_{i+1}$ of $\free_n$ which is $\Delta_{i+1}$-faithful and a strictly coherent book $\beta_{i+1}=\beta_i\cup\{g_{i+1}\mapsto s_{i+1}(g_{i+1})\}$.
\end{itemize}
For each $n$, the state $s_n$ agrees with $s_{n-1}$ over $\Phi_{n-1}$. Thus, for all $n_0$ and for all $n>n_0$, $s_n(g_{n_0})$ always attains the same value. In particular, for all $n,m\in \mathbb{N}$, $s_n(f)=s_m(f)$ for all $f\in \Phi$.


\begin{thm}\label{main2}
Let $\Phi$ be a finite subset of $\free_n$ and let $\beta$ be a book on $\Phi$. Then the following conditions are equivalent:
\begin{enumerate}
\item $\beta$ is strictly coherent;
\item $\beta$ extends to a faithful state of $\free_n$.
\end{enumerate}
\end{thm}
\begin{proof}
The direction $(2) \Rightarrow (1)$ is trivial. Thus, let $\beta$ be strictly coherent and fix an enumeration $g_1,g_2, \ldots$ of $\free_n\setminus\{\Phi, 0,1\}$.  The  construction above determines subsets  $\Phi=\Phi_0\subseteq\Phi_1\subseteq \Phi_2\subseteq\ldots$ of $\free_n$ and a sequence $\{s_i\}_{i\geq 1}$ of states of $\free_n$ such that:
\begin{itemize}
\item[(i)] Each $s_i$ is a $\Delta_i$-faithful state of $\free_n$;
\item[(ii)] For all $n>m$, $s_m(e)=s_n(e)$ for all $e\in \Phi_m$. 
\end{itemize}
 By construction of the $\Phi_i$'s, for every $f\in \free_n$ there exists an $m\geq 0$ such that $f\in \Phi_m$ and hence, by (ii), $s_m(f)=s_n(f)$ for all $n>m$. Therefore, $\{s_i(f)\}_{i\geq 0}$ is a Cauchy sequence. This gives that $\{s_i\}_{i\geq 0}$ is pointwise convergent. Define $s:\free_n\to[0,1]$ as follows: for each $f\in \free_n$,
$$
s(f)=\lim_{i\to\infty} s_i(f).
$$
Let us prove that $s$ is a state. Clearly $s(1)=1$. If $a\oplus b=0$ then, for all $i\geq 0$, $s_i(a\oplus b)=s_i(a)+s_i(b)$ and by the continuity of $+$, $s(a\oplus b)=\lim_{i\to\infty}s_i(a\oplus b)=\lim_{i\to\infty}s_i(a)+s_i(b)=\lim_{i\to\infty}s_i(a)+\lim_{i\to\infty}s_i(b)=s(a)+s(b)$. By construction, $s$ extends $\beta$ since so does each $s_n$. There remains to be proved that $s$ is faithful. We will provide two proofs of this fact.
\vspace{.2cm}

\noindent(Proof 1). By \cite[Theorem 5.2]{FHM18} $s$ is faithful iff for each finite subset $\Psi$ of $\free_n$ the restriction of $s$ to $\Psi$ is strictly coherent. Recalling the above construction, let $i_0$ be the minimum index such that 
$\Psi\subseteq\Phi_{i_0}$. Thus, the restriction of $s$ to $\Psi$ coincides with the restriction of $s_{i_0}$ to $\Psi$ and the restriction of $s_{i_0}$ to $\Phi_{i_0}$ is strictly coherent. The claim immediately follows because strict coherence is preserved for books contained in a strictly coherent one.
\vspace{.2cm}

\noindent(Proof 2). Let $1>f>0$. If $f\in \Phi$ there is nothing to prove. Conversely, assume that $f=g_{i}$ for some $i$. Therefore, for all $j\geq i$, $s_j(f)=\alpha>0$. Thus, $s(f)=\lim_{i\to\infty}s_i(f)=\alpha>0$ and the claim is settled.
\end{proof}

\section{Coherence, strict coherence and provability in \luk\ logic.}\label{sec:log}
Propositional \luk\ logic ($\L$ in symbols) is the logical calculus having MV-algebras as its equivalent algebraic semantics. Formulas of \luk\ logic will be denoted by lower case Greek letter and $\L(m)$ will stand for the set of formulas in a language with $m$ propositional variables. A complete axiomatization of $\L$ can be found in \cite[Definition 4.3.1]{CDM}. A formula $\varphi$ is said to be a {\em theorem}, in symbols $\vdash\varphi$, if $\varphi$ can be deduced from the axioms of $\L$ and by its unique rule of modus ponens. A {\em theory} $\Theta$ is a deductively closed set of formulas. 
A theory $\Theta$ of $\L(m)$ is said to be {\em finitely axiomatizable} if for some (necessarily satisfiable) formula $\theta\in \L(m)$, $\Theta$ is the smallest theory of $\L(m)$ which contains $\theta$.

By Proposition \ref{prop1}, valuations of the \luk\ language $\L(m)$ are in one-one correspondence with homomorphisms of $\free_m$ to $[0,1]_{MV}$ as well as with points of the $m$-cube $[0,1]^m$. Thus, a formula $\varphi\in \L(m)$ is a {\em tautology} if   $h(\varphi)=1$ for all homomorphisms $h:\free_m\to[0,1]_{MV}$ iff the oneset of $f_\varphi$ coincides with $[0,1]^m$, where $f_\varphi$ is the unique McNaughton function determined by $\varphi$ \cite{MuConstructive}.

For every $\mathscr{X}\subseteq[0,1]^m$ and theory $\Theta$ we write
$$
{\rm Th}(\mathscr{X})=\{\psi\in \L(m)\mid (\forall x\in \mathscr{X})\;f_\psi(x)=1\}
$$
and 
$$
{\rm Mod}(\Theta)=\{x\in [0,1]^m\mid (\forall \psi\in \Theta)\; f_\psi(x)=1\}.
$$
Given two (not necessarily finitely axiomatizable) theories $\Theta_1$ and $\Theta_2$, we  write $\Theta_1\models\Theta_2$, if ${\rm Mod}(\Theta_1)\subseteq{\rm Mod}(\Theta_2)$. 

Following \cite[Definition 3.9]{Mu12}, two rational polyhedra $\mathscr{P}$ and $\mathscr{Q}$ of $[0,1]^m$ are said to be {\em $\mathbb{Z}$-homeomorphic} (in symbols, $\mathscr{P}\cong_{\mathbb{Z}}\mathscr{Q}$) if there exists a homeomorphism $\eta:\mathscr{P}\to\mathscr{Q}$ such that both $\eta$ and $\eta^{-1}$, as maps from $\mathbb{R}^m\to\mathbb{R}^m$ are {\em $\mathbb{Z}$-maps}, i.e., $\eta$ and $\eta^{-1}$ are piecewise linear with integer coefficients.
\begin{lem}[{\cite[Theorem 3.20]{Mu12}}]\label{lemmaMund1}
For every $m=1,2,\ldots$, the pair $({\rm Th}, {\rm Mod})$ establishes a Galois connection between rational polyhedra of $[0,1]^m$ and finitely axiomatizable theories of $\L(m)$. In particular:
\begin{enumerate}
\item For every finitely axiomatizable theory $\Theta$ of $\L(m)$, there exists a unique rational polyhedron $\mathscr{P}_\Theta$ of $[0,1]^m$ such that ${\rm Mod}({\rm Th}(\Theta))\cong_\mathbb{Z}\mathscr{P}_\Theta$.
\item For each rational polyhedron $\mathscr{P}$ of $[0,1]^m$ there exists a unique finitely axiomatizable theory $\Theta_\mathscr{P}$ such that ${\rm Mod}({\rm Th}(\Theta_\mathscr{P}))\cong_\mathbb{Z}\mathscr{P}$.
\item For $\mathscr{P}_1$ and $\mathscr{P}_2$ rational polyhedra, $\mathscr{P}_1\subseteq \mathscr{P}_2$
iff $\Mod(\Theta_{\mathscr{P}_1})\subseteq\Mod(\Theta_{\mathscr{P}_2})$
iff $\Theta_{\mathscr{P}_1}\models\Theta_{\mathscr{P}_2}$.
\end{enumerate}
\end{lem}
In the rest of this section we will adopt the notation used in Lemma \ref{lemmaMund1} above with the following  exception: if $x\in ([0,1]\cap\mathbb{Q})^m$, we denote by $\Theta_x$ the finitely axiomatizable theory $\Theta_{\{x\}}$.

Let $\Phi$ be a subset of $\free_n$ of finite cardinality $k$ and let $\beta$ be a rational-valued book on $\Phi$. As noted at the beginning of Section \ref{sec:geoStrict}, $\{\beta\}$, $\mathscr{C}_\Phi$ and $\relb\;\mathscr{C}_\Phi$ are rational polyhedra of $[0,1]^k$. By Lemma \ref{lemmaMund1}, $\Theta_\beta$, $\Theta_{\mathscr{C}_\Phi}$ and $\Theta_{(\relb\;\mathscr{C}_\Phi)}$ are finitely axiomatizable.

The following lemma provides a first characterization of coherence and strict coherence in terms of deducibility. 
\begin{lem}\label{lemmaMain2}
Let $\Phi$ be a finite subset of $\free_n$ and let $\beta$ be a rational-valued book on $\Phi$. Thus the following conditions hold:
\begin{enumerate}
\item $\beta$ is coherent iff  $\Theta_\beta\models \Theta_{\mathscr{C}_\Phi}$.
\item $\beta$ is strictly coherent iff $\Theta_\beta\models \Theta_{\mathscr{C}_\Phi}$ and $\Theta_\beta\not\models \Theta_{(\relb\;\mathscr{C}_\Phi)}$.
\end{enumerate}
\end{lem}
\begin{proof}
(1)  follows from \cite[Theorem 2.1]{MuBookmaking}, Lemma \ref{lemmaMund1} and the definition of $\models$. Indeed, $\beta$ is coherent iff $\beta\in \mathscr{C}_\Phi$ iff  $\{\beta\}\subseteq\mathscr{C}_\Phi$.

As for (2), Theorem \ref{thm:MainGeo} shows that $\beta$ is strictly coherent iff $\beta\in \relint \mathscr{C}_\Phi=\mathscr{C}_\Phi\setminus\relb\;\mathscr{C}_\Theta$ iff $\beta\in \mathscr{C}_\Phi$ and $\beta\not \in\relb\;\mathscr{C}_\Theta$ iff $\Theta_\beta\models \Theta_{\mathscr{C}_\Phi}$ (by (1) above) and $\{\beta\}\not\subseteq\relb\;\mathscr{C}_\Theta$. By Lemma \ref{lemmaMund1} (3) this condition is equivalent to $\Theta_\beta\not\models\Theta_{(\relb\;\mathscr{C}_\Phi)}$.
\end{proof}

To characterize coherence and strict coherence in terms of  provability in \luk\ logic, we prepare. 

\begin{prop}\label{propConvComp}
There exists an effective procedure $\Pi$ to compute, for each rational polytope $\mathscr{P}$ of $[0,1]^k$, a formula $\Pi_\mathscr{P}$ which axiomatizes $\Theta_\mathscr{P}$.
\end{prop}
\begin{proof}
First, compute a regular complex $\Delta$ supporting $\mathscr{P}$ (see \cite[Chapter  6.2.2. and Theorem 6.5]{Desiderata}).
Notice that, $\ext \; \mathscr{P}\subseteq V(\Delta)$ and let  $\hat{h}_1, \ldots, \hat{h}_q$ be the normalized Schauder hats at the vertices $v_1,\ldots, v_q$ of $\Delta$. For $j=1,\ldots, q$ let $\Pi_j$ be the \luk\ formulas computed from $\hat{h}_j$ (see \cite{MuConstructive}). Let further, 
$$
\Pi_\mathscr{P}=\bigoplus_{j=1}^q \Pi_j. 
$$
Since each $\hat{h}_j$ is a member of $\L(k)$, $\Pi_\mathscr{P}$ belongs to $\L(k)$. There remains to be proved that $\Pi_\mathscr{P}$ axiomatizes $\Theta_\mathscr{P}$. To this purpose, let us prove that 
\begin{center}
$x\in \mathscr{P}$ iff $h_x(\Pi_\mathscr{P})=1$.
\end{center}
As a matter of fact, by Lemma \ref{lemmaSchauder} (3) the oneset of the McNaughton function $\Pi_\mathscr{P}$ is $\mathscr{P}$. Thus the claim is settled.
\end{proof}
\begin{cor}\label{lemmaCompute}
There exists an effective procedure $\Pi$ which computes, for each  $\Phi=\{f_1,\ldots, f_k\}\subseteq \free_n$  and for each $\beta\in [0,1]^k$,  formulas $\Pi_\Phi$, $\Pi_{(\relb\; \Phi)}$ and $\Pi_\beta$ of $\L(k)$ which respectively axiomatize $\Theta_{\mathscr{C}_\Phi}$, $\Theta_{(\relb\; \mathscr{C}_\Phi)}$ and $\Theta_\beta$. 
\end{cor}
\begin{proof}
As the reader will recall from Section \ref{sec:geoStrict}, $\mathscr{C}_\Phi$ is a polytope. Thus, $\mathscr{C}_\Phi$, $\{\beta\}$ are rational polytopes of $[0,1]^k$, whence $\Pi_\Phi$ and $\Pi_\beta$ are computed as in Proposition \ref{propConvComp}. 

As for $\Pi_{(\relb\;\Phi)}$,  $\relb\; \mathscr{C}_\Phi$ is not convex. However, it can be realized as the finite union of the faces $F_1,\ldots, F_l$ of $\mathscr{C}_\Phi$. Each face $F_i$ is a rational  polytope, whence Proposition \ref{propConvComp} yields \luk\ formulas $\Pi_1,\ldots, \Pi_l$ such that $x\in F_i$ iff $h_x(\Pi_i)=1$. Thus, let 
$$
\Pi_{(\relb\;\Phi)}=\bigvee_{i=1}^l\Pi_i.
$$
Finally, $x\in \relb\;\mathscr{C}_\Phi$ iff exists $i=1,\ldots, n$ such that $x\in F_i$ iff $h_x(F_i)=1$ iff $h_x(\Pi_{(\relb\;\Phi)})=1$. 
\end{proof}

In the light of the above corollary, we may write $\Pi_\Phi$, $\Pi_{(\relb\;\Phi)}$ and $\Pi_\beta$ without danger of confusion. In the following characterization, for every \luk\ formula $\psi$, we write $\psi^n$  for $\psi\odot\ldots\odot\psi$ ($n$-times).
\begin{thm}\label{thm:logic}
Let $\Phi$ be a finite set of $\free_n$ and let $\beta$ be a book on $\Phi$. Then the following conditions hold:
\begin{enumerate}
\item $\beta$ is coherent iff there exists a non-zero $n\in \mathbb{N}$ such that $\vdash (\Pi_\beta)^n\to \Pi_\Phi$.
\item $\beta$ is strictly coherent iff there exists a non-zero $n\in \mathbb{N}$ such that $\vdash (\Pi_\beta)^n\to \Pi_\Phi$ and for all non-zero $n\in \mathbb{N}$, $\not\vdash (\Pi_\beta)^n\to \Pi_{(\relb\;\Phi)}$.
\end{enumerate}
\end{thm}
\begin{proof}
Both claims follow from Lemma \ref{lemmaMain2}, Corollary \ref{lemmaCompute}, the completeness theorem of \luk\ calculus and \luk\ deduction theorem stating that $\varphi\vdash\psi$ iff there exists a non-zero $n\in \mathbb{N}$ such that $\vdash\varphi^n\to\psi$ (see \cite[\S4]{CDM}). We will prove (1) since the proof of (2) is essentially the same.

By Lemma \ref{lemmaMain2}, $\beta$ is coherent iff $\Theta_\beta\models\Theta_{\mathscr{C}_\Phi}$ iff (from Lemma \ref{lemmaCompute}) $\Pi_\beta\models\Pi_\Phi$. The completeness theorem of \luk\ calculus shows that $\Pi_\beta\models\Pi_\Phi$ iff $\Pi_\beta\vdash\Pi_\Phi$ iff  $\vdash (\Pi_\beta)^n\to\Pi_\Phi$ for some $n>0$.
\end{proof}

\section{Conclusion}

In this paper we have presented geometrical, measure-theoretical and  logical characterizations for the strict coherence of books on \luk\ infinite-valued events. Our first result shows that, for any finite subset $\Phi$ of a finitely generated free MV-algebra ${\bf A}$, the set of all strictly coherent books on $\Phi$ coincides with the relative interior of the polytope of all coherent ones; the second characterization is a de Finetti-like theorem: a book on $\Phi$ is strictly coherent if and only if it extends to a faithful state of ${\bf A}$. Finally, our last theorem gives a characterization of coherence and strict coherence in terms of the provability relation of propositional \luk\ logic.   

We believe that this last result is interesting both from the logical and philosophical perspective as it may shed a light on an  {\em intuitive reading} of propositional \luk\ logic.  Specifically, it is of particular interest to put forward a comparison between the role of \luk\ logic prompted by Theorem \ref{thm:logic} in theories of uncertain reasoning and the semantics proposed in \cite{Marra}. There, the author, investigating the problem of  {\em artificial precision} in theories of vagueness based on real numbers as degrees of truth,  presents \luk\ logic as a suitable formal system to handle vague predicates\footnote{The author wishes to thank Eduardo Barrio for pointing out this to him.}.

In our future work we will mainly focus on extending the results of this paper to more general algebraic structures. Particularly promising seems to be the class of {\em finitely presented} MV-algebras (see \cite{MaSpa} and \cite[Theorem 6.3]{Mu12}). Further, we will address the problem of determining an NP-algorithm to check strict coherence for \luk\ events. The solution of this  problem would immediately yield that for each finite set  $\Phi$ of \luk\ events, $\mathscr{K}_\Phi^\mathbb{Q}$ is NP-complete (see \cite[\S 7]{FHM18}).
\vspace{.2cm}

\noindent{\bf Acknowledgement}. The author is grateful to the two referees for their careful reading and valuable comments. He is also much indebted to Lluis Godo, Hykel Hosni and Daniele Mundici for several conversations on the subject of this paper. The author acknowledges partial support by the Spanish Ram\'on y Cajal research program RYC-2016-19799; the Spanish FEDER/MINECO project TIN2015- 71799-C2-1-P and the SYSMICS project (EU H2020-MSCA-RISE-2015, Project 689176).

\end{document}